\documentclass[12pt,a4paper]{amsart}
\usepackage[headings]{fullpage}
\usepackage{amssymb}
\usepackage{amsthm}

\newcommand{\Qp}{\mathbf{Q}_p}
\newcommand{\Cp}{\mathbf{C}_p}
\newcommand{\Zp}{\mathbf{Z}_p}
\newcommand{\ZZ}{\mathbf{Z}}
\newcommand{\NN}{\mathbf{N}}
\newcommand{\QQ}{\mathbf{Q}}
\newcommand{\kbf}{\mathbf{k}}

\newcommand{\Gm}{\mathbf{G}_\mathrm{m}}
\newcommand{\nequiv}{\not\equiv}
\newcommand{\MM}{\mathfrak{m}}
\newcommand{\val}{\operatorname{val}}
\newcommand{\Hom}{\operatorname{Hom}}
\newcommand{\vp}{\val_p}
\newcommand{\dcroc}[1]{[\![ #1 ]\!]}
\newcommand{\LT}{\mathrm{LT}}

\theoremstyle{plain}

\newtheorem{theo}{Theorem}[section]
\newtheorem{coro}[theo]{Corollary}
\newtheorem{lemm}[theo]{Lemma}
\newtheorem{prop}[theo]{Proposition}
\newtheorem{defi}[theo]{Definition}

\newtheorem*{theoA}{Theorem A}

\begin{document}

\title{$p$-adic Fourier theory for $\mathbf{Q}_{p^2}$ and the Monna map}

\author{Konstantin Ardakov}
\address{Konstantin Ardakov \\ 
Mathematical Institute \\ 
University of Oxford}
\email{ardakov@maths.ox.ac.uk}
\urladdr{http://people.maths.ox.ac.uk/ardakov/}

\author{Laurent Berger}
\address{Laurent Berger \\ 
UMPA, ENS de Lyon \\ 
UMR 5669 du CNRS}
\email{laurent.berger@ens-lyon.fr}
\urladdr{https://perso.ens-lyon.fr/laurent.berger/}

\begin{abstract}
We show that the coefficients of a power series occurring in $p$-adic Fourier theory for $\QQ_{p^2}$ have valuations that are given by an intriguing formula.
\end{abstract}

\date{\today}

\maketitle

\setlength{\baselineskip}{18pt}

\section*{Introduction}

Let $L$ be a finite extension of $\Qp$, let $\pi$ be a uniformizer of $o_L$ and let $\LT$ be the Lubin-Tate formal $o_L$-module attached to $\pi$. The formal group maps over $o_{\Cp}$ from $\LT$ to $\Gm$ play an important role in $p$-adic Fourier theory (see \cite{ST}). Choose a coordinate $Z$ on $\LT$, and let $G(Z) \in o_{\Cp} \dcroc{Z}$ be a generator of $\Hom_{o_{\Cp}}(\LT,\Gm)$, so that 
\[ G(Z) = \sum_{k \geq 1} P_k(\Omega) \cdot Z^k = \exp(\Omega \cdot \log_{\LT}(Z))-1 \] for a certain element $\Omega \in o_{\Cp}$ and polynomials $P_k(Y) \in L[Y]$. We have (\S 3 of \cite{ST}) $\vp(\Omega) = 1/(p-1)-1/e(q-1)$ where $e$ is the ramification index of $L$ and $q = |{o_L/\pi o_L}|$. The power series $G(Z)$ gives rise to a function on $\MM_{\Cp}$ and the theory of Newton polygons then allows us to compute the valuation of $P_k(\Omega)$ for $k=q^j/p^{\lfloor (j-1)/e \rfloor+1}$ with $j \geq 0$ (Theorem 1.5.2 of \cite{AB24}). However, the valuation of $P_k(\Omega)$ for most $k \geq 2$ has no geometric significance and depends on the choice of the coordinate $Z$. 

During our work on the character variety, we computed the valuation of $P_k(\Omega)$ for many small values of $k$ in a special case: we took $L=\QQ_{p^2}$ and $\pi=p$ and chose a coordinate $Z$ on $\LT$ for which $\log_{\LT}(Z) = \sum_{m \geq 0} Z^{q^m}/p^m$ (this is possible by \S 8.3 of \cite{Haz}). Note that in this setting, the theory of Newton polygons gives $\vp(P_k(\Omega))$ precisely when $k$ is a power of $p$. Let $w : \ZZ_{\geq 0} \to \QQ$ be the map defined by  \[ w(k) =  \frac{p}{q-1} \cdot (k_0+p^{-1} k_1 + \cdots + p^{-h} \cdot k_h) \text{ if $k=(k_h \cdots k_0)_p$ in base $p$.} \]
For all $k$ for which we were able to compute $\vp(P_k(\Omega))$, we found that $\vp(P_k(\Omega)) = w(k)$. The main result of this note is that this formula holds for all $k$.
\begin{theoA}
For all $k \geq 1$, we have $\vp(P_k(\Omega)) = w(k)$.
\end{theoA} 
The proof involves a careful study of the functional equation that $G(Z)$ satisfies, and a direct computation of $\vp(P_k(\Omega))$ for small values of $k$. The function $w$ is related to the Monna map, defined in \cite{M52}.
 
\section{The polynomials $P_m(Y)$}

Let $L=\QQ_{p^2}$ and $\pi=p$, so that $q=p^2$, and choose a coordinate $Z$ on $\LT$ for which $\log_{\LT}(Z) = \sum_{k \geq 0} Z^{q^k}/p^k$. The polynomials $P_m(Y) \in L[Y]$ are given by \[ \exp(Y \cdot \log_{\LT}(Z)) = \sum_{m=0}^{+\infty} P_m(Y) \cdot Z^m. \] 

\begin{prop}
\label{prop11}
We have
\[ P_m(Y) = \sum_{m_0+qm_1+\cdots+q^d m_d=m} \frac{Y^{m_0+\cdots+m_d}}{m_0! \cdots m_d! \cdot p^{1 \cdot m_1 + 2 \cdot m_2 + \cdots + d \cdot m_d}} \]
\end{prop}

\begin{proof}
Since $\log_{\LT}(Z) = \sum_{k \geq 0} Z^{q^k}/p^k$ and $\exp$ is the usual exponential, 
\[ \sum_{m=0}^{+\infty} P_m(Y) Z^m = \exp(Y \cdot \log_{\LT}(Z)) = \prod_{k \geq 0} \exp( Y \cdot Z^{q^k}/p^k) = \prod_{k \geq 0} \sum_{j \geq 0} ( Y \cdot Z^{q^k}/p^k)^j / j! \]
The coefficient of $Z^m$ is the sum of $Y^{m_0+\cdots+m_d} / m_0! \cdots m_d! \cdot p^{1 \cdot m_1 + 2 \cdot m_2 + \cdots + d \cdot m_d}$ over all $d \geq 0$ and $(m_0,\cdots,m_d) \in \ZZ_{\geq 0}^{d+1}$ such that $m_0+qm_1+\cdots+q^d m_d=m$. 
\end{proof}

For example, if $i \leq q-1$, then 
\begin{align*}
P_i(Y) & =Y^i/i! \\
P_{q+i}(Y) &= \frac{Y^{q+i}}{(q+i)!} + \frac{Y^{i+1}}{p \cdot i!} \\ 
P_{2q+i}(Y) &= \frac{Y^{2q+i}}{(2q+i)!} + \frac{Y^{q+i+1}}{p \cdot (q+i)!} + \frac{Y^{i+2}}{2p^2 \cdot i!}.
\end{align*}

Because $L = \QQ_{p^2}$, it follows from Lemma 3.4.b of \cite{ST} that 
\[\vp(\Omega) = \frac{1}{p-1} - \frac{1}{e(q-1)} = \frac{p}{q-1}.\]

\begin{lemm}
\label{rema15}
If $i \leq q-1$ and $i=(ab)_p$ in base $p$, then $\vp(P_i(\Omega)) = \frac{a+bp}{q-1} = w(i)$.
\end{lemm}

\begin{proof}
If $i \leq q-1$, then $P_i(\Omega) = \Omega^i / i!$ by Proposition \ref{prop11}, so that
\[ \vp(P_i(\Omega)) = i \cdot \left( \frac{1}{p-1}-\frac{1}{q-1} \right)-\frac{i-s_p(i)}{p-1} = \frac{s_p(i)}{p-1}-\frac{i}{q-1} = \frac{a+bp}{q-1}. \qedhere \] 
\end{proof}

\section{The map $w$}

Recall that $w : \ZZ_{\geq 0} \to \QQ$ is the map defined by  \[ w(k) =  \frac{p}{q-1} \cdot (k_0+p^{-1} k_1 + \cdots + p^{-h} \cdot k_h) \text{ if $k=(k_h \cdots k_0)_p$ in base $p$.} \]

\begin{prop}
\label{prop17}
The function $w: \ZZ_{\geq 0} \to \QQ_{\geq 0}$ has the following properties:
\begin{enumerate}
\item $w(k) < 1+1/(q-1)$;
\item $w(k) \geq 1$ if and only if $k \equiv -1 \bmod{q}$, and then $w(k)>1$ unless $k=q-1$;
\item if $\ell>k$, then $w(\ell) - w(k) \in \ZZ$ if and only if $k=qj$ and $\ell=qj+(q-1)$;
\item $w(pk) = 1/p \cdot w(k)$;
\item $w(p^n k+i) = w(p^n k) +w(i)$ if $0 \leq i \leq p^n-1$;
\item For all $a, b \geq 0$ we have $w(a+b) \leq w(a) + w(b)$.
\end{enumerate}
\end{prop}

\begin{proof}
Item (1) results from the fact that\[ w(k) = (k_0+p^{-1} k_1 + \cdots + p^{-h} \cdot k_h) \cdot \frac{p}{q-1} < \frac{p^2}{q-1} = 1+\frac{1}{q-1}.  \] 
If $k_0 \leq p-2$, or if $k_0 = p-1$ and $k_1 \leq p-2$, then $w(k) \leq (p^{h+1}-1-p^{h-1})/p^{h-1}(q-1) < 1$, so if $w(k) \geq 1$, then $k_0 = p-1$ and $k_1 = p-1$, and  $k \equiv -1 \bmod{q}$. Conversely, if $k \equiv -1 \bmod{q}$, then $k_0 = p-1$ and $k_1 = p-1$, and $w(k) \geq 1$. Finally, if we have equality, then $k_i=0$ for all $i \geq 2$. This proves (2).

Write $k=(k_h \cdots k_0)_p$ and $\ell = (\ell_i \cdots \ell_0)_p$. Since $w(k) < 1+1/(q-1)$, if $w(\ell) - w(k) \in \ZZ_{\geq 0}$, then $w(\ell)=w(k)$ or $w(\ell)=w(k)+1$. If $w(\ell)=w(k)$, then $k_0+p^{-1} k_1 + \cdots + p^{-h} \cdot k_h = \ell_0 + p^{-1} \ell_1 + \cdots + p^{-i} \cdot \ell_i$. By comparing $p$-adic valuations, we get $h=i$, and then $k_h \equiv \ell_i \bmod{p}$ so that $k_h = \ell_i$. By descending induction, $k_j = \ell_j$ for all $j$, and $k=\ell$. If $w(\ell)=w(k)+1$, then $w(\ell) \geq 1$, and hence $\ell = (\ell_i \cdots \ell_2 (p-1)_1(p-1)_0)_p$ by item (2). We then have $w((\ell_i \cdots \ell_2 0_10_0)_p) = w(k)$ and hence $k=(\ell_i \cdots \ell_2 0_10_0)_p$. This implies (3).

Items (4) and (5) are straightforward. For item (6), let $\{a_i\}$, $\{b_i\}$ and $\{c_i\}$ be the digits of $a$, $b$ and $c$ in base $p$. Let $r_0 = 0$ and let $r_i \in \{0,1\}$ be the $i$th carry when adding $a$ and $b$, so that $c_i = a_i + b_i +r_i - p r_{i+1}$. The result follows from the following computation.
\[ \sum_{i \geq 0} \frac{c_i}{p^i} = \sum_{i \geq 0} \frac{a_i+b_i}{p^i} + \frac{r_i}{p^i} - \frac{p r_{i+1}}{p^i} = \sum_{i \geq 0} \frac{a_i+b_i}{p^i} -  (p^2 - 1) \sum_{i \geq 1} \frac{r_i}{p^i} \leq \sum_{i \geq 0} \frac{a_i+b_i}{p^i}. \qedhere \]
\end{proof}

\section{Congruences for the $P_k(\Omega)$}

From now on, we write $u_k$ for $P_k(\Omega)$ to lighten the notation. Recall that $q=p^2$.
The power series $G(Z)$ is a map between $\LT$ and $\Gm$, so that $G([p]_{\LT} (Z)) = [p]_{\Gm}(G(Z))$. 

\begin{prop}
\label{prop13}
We have $\sum_{m=1}^{+\infty} u_m Z^{qm} \equiv \sum_{k=1}^{+\infty} u_k^p Z^{kp} \bmod{p \cdot \MM_{\Cp}}$.
\end{prop}

\begin{proof}
We have $G(Z) \in \MM_{\Cp} \dcroc{Z}$ and $[p]_{\LT} (Z) \equiv Z^q \bmod{p}$ and $[p]_{\Gm}(Z) = Z^p \bmod{p}$. 

Since $G([p]_{\LT} (Z)) = [p]_{\Gm}(G(Z))$, we get $G(Z^q) \equiv G(Z)^p \bmod{p \cdot \MM_{\Cp}}$.
\end{proof}

\begin{coro}
\label{coro114}
If $k$ is not divisible by $p$, then $\vp(u_k) > 1/p$.
\end{coro}

\begin{coro}
\label{ppm}
We have $u_{pm}^p \equiv u_m \bmod{p \cdot \MM_{\Cp}}$.
\end{coro}

\begin{proof}
Take $k=pm$ in Proposition \ref{prop13}.
\end{proof}

\begin{coro} 
\label{coro14} 
Take $m\geq 0$.
\begin{enumerate}  
\item Suppose that $\vp(u_{m}) \leq 1$. Then $\vp(u_{pm})=1/p \cdot \vp(u_m)$.
\item Suppose that $\vp(u_{m}) > 1$. Then $\vp(u_{pm}) > 1/p$.
\end{enumerate}
\end{coro}

\begin{proof} 
Both cases follow easily from Corollary \ref{ppm}.
\end{proof}

We now compare $[p]_{\LT}(Z)$ and $Z^q+pZ$ (compare with (iv) of \S 2.2 of \cite{Haz}).

\begin{lemm}
\label{prop118}
We have $[p]_{\LT}(Z) = Z^q + p Z + p^2 \cdot s(Z)$ for some $s(Z) \in Z^2 \cdot \Zp\dcroc{Z}$.
\end{lemm}

\begin{proof}
There exists $r(Z) \in Z^2 \cdot \Zp\dcroc{Z}$ such that $[p]_{\LT}(Z) = Z^q + pZ + p r(Z)$. By the properties of $\log_{\LT}$, we have $\log_{\LT}([p]_{\LT}(Z)) = p \log_{\LT}(Z)$. Expanding around $Z^q$, we get
\[ \log_{\LT}( Z^q + pZ + p r(Z) ) = \log_{\LT}(Z^q) + (pZ + p r(Z))  \log'_{\LT}(Z^q) + \sum_{i \geq 2} \frac{(pZ + p r(Z))^i}{i!}  \log^{(i)}_{\LT}(Z^q) \]
Our choice of $\log_{\LT}$ is such that $\log_{\LT}(Z^q) = p \log_{\LT}(Z)-pZ$ and $\log'_{\LT}(Z) \in 1 + pZ \cdot \Zp\dcroc{Z}$ and $\log_{\LT}^{(i)}(Z) \in p \Zp\dcroc{Z}$ for all $i \geq 2$. Note also that $p^{i+1}/i! \in p^2 \Zp$ for all $i \geq 2$. 

The above equation now implies that $pr(Z) \equiv 0 \bmod{p^2}$ so that $r(Z) = ps(Z)$.
\end{proof}

\begin{coro}
\label{coro119} 
The coefficient of $Z^{qn}$ in $G([p]_{LT}(Z))$ is congruent to $u_n  \bmod{p^2}$.
%We have $G([p]_{\LT}(Z)) \equiv \sum_{k \geq 1} u_k Z^{qk} + \sum_{m \geq 1} p m \cdot u_m Z^{q(m-1)+1} \bmod{p^2}$.
\end{coro}

\begin{proof}
Since $[p]_{\LT}(Z) \equiv  Z^q + p Z \bmod{p^2}$, Lemma \ref{prop118} tells us that
\begin{align*} 
G([p]_{\LT}(Z)) & \equiv G(Z^q) + p Z \cdot G'(Z^q) \bmod{p^2} \\
& \equiv \sum_{k \geq 1} u_k Z^{qk} + \sum_{m \geq 1} p m \cdot u_m Z^{q(m-1)+1} \bmod{p^2}.  
\end{align*}
Hence $pZ \cdot G'(Z^q)$ doesn't contribute to the coeffiicent of $Z^{qn}$ modulo $p^2$.
\end{proof}

\begin{prop}
\label{DivByu1} 
For all $k \geq 1$, we have  $k \cdot u_k  = u_1 \cdot \sum_{r=0}^{\lfloor \log_q(k) \rfloor} p^r u_{k-q^r}$.
\end{prop}

\begin{proof} 
We have $\sum_{k \geq 0} u_k Z^k = \exp(u_1 \cdot \log_{\LT}(Z))$. Applying $d/dZ$, we get 
\begin{align*} 
\sum_{k \geq 1} k u_k Z^{k-1} & = \exp(u_1 \cdot \log_{\LT}(Z)) \cdot u_1 \cdot \log_{\LT}'(Z) \\
&  = u_1 \cdot (\sum_{i \geq 0} u_i Z^i) \cdot (\sum_{r \geq 0} (q/p)^r Z^{q^r-1}).
\end{align*} 
The result follows from looking at the coefficient of $Z^{k-1}$ on both sides.
\end{proof}

\begin{coro}
\label{kuk} 
We have $u_1 \cdot u_{k-1} \equiv k u_k \bmod p$ for all $k \geq 1$.
\end{coro}

\begin{prop}
\label{Modp^2} 
If $0 \leq i \leq p-1$ and $m \geq p$, then there exists $\zeta_{i,m} \in o_L$ such that
\[ u_{mp+i} \equiv \binom{mp+i}{i}^{-1} \cdot u_{mp} \cdot u_i + p \cdot \zeta_{i,m} \cdot u_{p(m-p) + i + 1}  \bmod{p^2}. \]
\end{prop}

\begin{proof} 
We proceed by induction on $i$. When $i = 0$, we can even achieve equality by setting $\zeta_{0,m} := 0$, because $u_0 = 1$. Write $k := mp+i$ for brevity. For $i \geq 1$ we have
\[ u_k  \equiv  \frac{1}{k} u_1 \cdot u_{k-1}  +  \frac{p}{k} u_1 \cdot u_{k-q}  \bmod{p^2} \]
by Proposition \ref{DivByu1}, because here $k \in o_L^\times$. By the inductive hypothesis, we have
\[ u_{k-1} \equiv \binom{k-1}{i-1}^{-1} u_{mp} \cdot u_{i-1}  +  p \zeta_{i-1,m} \cdot u_{k-q}   \bmod{p^2}.\]
Note that since $i \leq p-1$, we have $u_i = u_1^i/i!$ by Proposition \ref{prop11}, so $u_1 u_{i-1} = \frac{u_1^i}{(i-1)!} = i u_i$. Substituting this information, we obtain
\begin{align*} 
u_k & \equiv \frac{u_1}{k}\cdot\left(\binom{k-1}{i-1}^{-1} u_{mp} \cdot u_{i-1} + p \zeta_{i-1,m}u_{k-q}\right) + \frac{p}{k} u_1 \cdot u_{k-q} \\
& \equiv \frac{i}{k}\binom{k-1}{i-1}^{-1}u_{mp}\cdot u_i+\frac{p}{k}(\zeta_{i-1,m}+1)u_1\cdot u_{k-q} \bmod{p^2}.
\end{align*}
On the other hand, by Corollary \ref{kuk}, we have
\[ p u_1 \cdot u_{k-q} \equiv p (k-q+1)u_{k-q+1}   \bmod{p^2}.\]
Hence we can rewrite the congruence as follows:
\[ u_k \equiv \binom{k}{i}^{-1} u_{mp}\cdot u_i+p\frac{k-q+1}{k}(\zeta_{i-1,m}+ 1) u_{k-q+1}   \bmod{p^2}.\]
Define $\zeta_{i,m} := \frac{k-q+1}{k}(\zeta_{i-1,m}+ 1)$ and observe that this lies in $o_L$ because $p \nmid k$. 
\end{proof}

We need to know what $\zeta_{p-1,m}$ is modulo $p$.

\begin{lemm}
\label{LastZeta} 
Take $1 \leq i \leq p-1$ and $m \geq 0$ and let $k = mp + i$. 

If $\zeta_{0,m} = 0$ and $\zeta_{i,m} = \frac{k-q+1}{k} (\zeta_{i-1,m} + 1)$ whenever $1 \leq i \leq p-1$, then $\zeta_{p-1,m} \equiv 0  \bmod{p}$.
\end{lemm}

\begin{proof} 
Note that modulo $p$, the recurrence relation satisfied by $\zeta_{i,m}$ is simply
\[ \zeta_{i,m} \equiv \frac{i+1}{i}(\zeta_{i-1,m} + 1)   \bmod{p}.\]
Now set $i = p-1$ to see that $\zeta_{p-1,m} \equiv 0  \bmod{p}$.
\end{proof}

\section{Proof of Theorem A}

We now use the functional equation of $G(Z)$ modulo $p^2$ in order to prove Theorem A.

\begin{defi} 
For each $n \geq 0$, let $C_n$  be the coefficient of $Z^{q n}$ in 
\[ (1 + G(Z))^p = \left(\sum_{k=0}^\infty u_k Z^k\right)^p. \]
\end{defi}

We develop some notation to compute $C_n$. 

\begin{defi} \hspace{1cm}
\begin{enumerate}
\item Let $|\mathbf{k}| := k_1 + \cdots + k_p$ for all $\mathbf{k} \in \NN^p$. 
\item For each $\mathbf{k} \in \NN^p$, define $u_{\mathbf{k}} := u_{k_1} \cdot u_{k_2} \cdot \cdots \cdot u_{k_p}$.
\item For each $n \geq 0$, let $X_n \subset \NN^p$ be a complete set of representatives for the orbits of the natural action of $S_p$ on $\{\mathbf{k} \in \NN^p : |\mathbf{k}| = n\}$. 
\end{enumerate}
\end{defi}

In this language, expanding $\left(\sum_{k=0}^\infty u_k Z^k\right)^p$ gives the following

\begin{lemm}
\label{CnFormula}  
We have $C_n = \sum_{\mathbf{k} \in X_{qn}} |S_p \cdot \mathbf{k}| \hspace{0.1cm} u_{\mathbf{k}}$.
\end{lemm}

\begin{lemm}
\label{Orbit} 
We have $\vp(|S_p \cdot \mathbf{k}|) = 1$ whenever $k_i \neq k_j$ for some $i \neq j$.
\end{lemm}

\begin{proof} 
Let $H$ be the stabiliser of $\mathbf{k}$ in $S_p$, so that $|S_p \cdot \mathbf{k}| = |S_p| / |H|$. If $k_i \neq k_j$ for some $i \neq j$, then $H$ cannot contain any $p$-cycle. The only elements of $S_p$ of order $p$ are $p$-cycles, so by Cauchy's Theorem, $\vp(|H|) = 0$. Hence $\vp(|S_p| / |H|) = \vp(|S_p|) = 1$.
\end{proof}

\begin{lemm}
\label{NotMultsOfp^2} 
If $\mathbf{k} \in X_{q n} \setminus q \NN^p$, then $\vp( u_{\mathbf{k}} ) > w(n) - 1$.
\end{lemm}

\begin{proof} 
Since $\frac{1}{q-1} > w(n) - 1$ by Proposition \ref{prop17}(1), it is enough to show that
\[ \vp(u_{\mathbf{k}}) > \frac{1}{q-1}.\]
If some $k_i$ is not divisible by $p$, then by Corollary \ref{coro114},
\[\vp(u_{\mathbf{k}}) \geq \vp(u_{k_i}) > \frac{1}{p} > \frac{1}{q-1}.\]
Assume now that for each $i = 1, \ldots, p$, we can write $k_i = p m_i$ for some $m_i \geq 0$ so that $|\mathbf{m}| = \frac{1}{p}|\mathbf{k}| = pn$. Since $\mathbf{k} \notin q \NN^p$ by assumption, we must have $m_i \nequiv 0  \bmod{p}$ for some $i$. Because $|\mathbf{m}| = np \equiv 0  \bmod{p}$, in this case there must be at least two distinct indices $i,j$ such that $m_i \neq 0  \bmod{p}$ and $m_j \neq 0  \bmod{p}$. Using Corollary \ref{coro114} again, we obtain
\[ \vp(u_{\mathbf{m}}) \geq \vp(u_{m_i}) + \vp(u_{m_j}) \geq \frac{2}{p} > \frac{p}{q-1}.\]

Suppose now that $\vp(u_{m_i}) \leq 1$ for all $i$. Then Corollary \ref{coro14}(1) implies that 
\[ \vp(u_{\mathbf{k}}) = \frac{1}{p} \vp(u_{\mathbf{m}})  > \frac{1}{p} \cdot \frac{p}{q-1} = \frac{1}{q-1}. \]
Otherwise, for at least one index $i$ we have $\vp(u_{m_i}) > 1$, and then Corollary \ref{coro14}(2) gives
\[ \vp(u_{\mathbf{k}}) \geq \vp(u_{k_i}) > \frac{1}{p} > \frac{1}{q-1}. \qedhere \]
\end{proof}

We can now prove Theorem A.

\begin{theo}
\label{thA}
We have $\vp(u_n) = w(n)$ for all $n \geq 0$.
\end{theo}

\begin{proof} 
We prove the stronger statement $\vp(u_n) = w(n) = p \cdot \vp(u_{pn})$ by induction on $n$. The base case $n = 0$ is clear, so assume $n \geq 1$. We first show that $\vp(u_n) = w(n)$.

Write $n = mp + i$ with $0 \leq i \leq p-1$. Then $\vp(u_i) = w(i)$ holds by Lemma \ref{rema15}. Since $n \neq 0$, we must have $m < n$ so $\vp(u_{mp}) = \frac{1}{p} w(m)$ by the inductive hypothesis. Using  (4) and (5) of Proposition \ref{prop17}, we see that
\[\vp(u_i u_{mp}) = \vp(u_i) + \vp(u_{mp}) = w(i) + \frac{1}{p}w(m) = w(pm+i) = w(n).\]
Suppose first that $n \nequiv -1  \bmod{q}$. Then $w(n) < 1$ by Proposition \ref{prop17}(2), which means that $\vp(u_i u_{mp}) = w(n) < 1$. By Proposition \ref{Modp^2}, we have
\[ u_n \equiv \binom{mp+i}{i}^{-1} u_i u_{mp}   \bmod{p}.\]
We have $\binom{mp+i}{i} \equiv 1  \bmod{p}$ by Lucas' theorem, and therefore $\vp(u_n) = w(n)$.

Suppose now that $n \equiv -1  \bmod{q}$. Then $i = p-1$, and Proposition \ref{Modp^2} tells us that
\[u_n \equiv \binom{n}{p-1}^{-1} u_{mp} \cdot u_{p-1} + p\zeta_{p-1,m} \cdot u_{n-q+1}   \bmod{p^2}.\]
We have $\zeta_{p-1,m} \equiv 0  \bmod{p}$ by Lemma \ref{LastZeta}. Hence in fact $u_n \equiv \binom{n}{p-1}^{-1} u_{mp} u_{p-1}  \bmod{p^2}$. Since $\vp(u_{mp}u_{p-1}) = w(n) < 2$ by Proposition \ref{prop17}(1), we again conclude that \[\vp(u_n) = \vp(u_{mp}) + \vp(u_{p-1}) = w(n).\]
To complete the induction step, we must show that $w(n) = p \vp(u_{pn}) = \vp(u_{pn}^p)$. In order to do this, we compare the coefficients of $Z^{qn}$ in the functional equation for $G(Z)$
\[ G([p]_{\LT}(Z)) = [p]_{\mathbf{G}_m}(G(Z)) = (1 + G(Z))^p - 1 \]
modulo $p^2$. Using Corollary \ref{coro119} and Lemma \ref{CnFormula}, we see that
\begin{equation} 
\tag{$\diamond$} 
u_n \equiv C_n  = \sum_{\mathbf{k} \in X_{qn}} |S_p \cdot \mathbf{k}| \hspace{0.1cm} u_{\mathbf{k}}  \bmod{p^2}.
\end{equation}
Define $\kbf_0 := (pn,pn, \cdots, pn)$. We will now proceed to show that in fact 
\begin{equation}
\tag{$\star$} 
\vp( |S_p \cdot \kbf| u_{\kbf}) > w(n) \quad \text{for all} \quad \kbf \in X_{qn} \setminus \{\kbf_0\}.
\end{equation}
Note that $w(n) < 2$ by Proposition \ref{prop17}(1) and that $u_{\mathbf{k}_0} = u_{pn}^p$. Hence congruence $(\diamond)$ together with $(\star)$ imply that $\vp(u_n-u_{np}^p) > w(n)$. Since we already know that $\vp(u_n) = w(n)$ this shows that $\vp(u_{np}^p) = \vp(u_n) = w(n)$ and completes the proof.  

Since at least two entries of $\kbf$ must be distinct when $\kbf \neq \kbf_0$, we have $\vp(|S_p\cdot \kbf|) = 1$ by Lemma \ref{Orbit}, so we're reduced to showing that
\begin{equation}
\tag{$\star\star$} 
\vp( u_{\kbf}) > w(n) - 1 \quad \text{for all} \quad \kbf \in X_{qn}\setminus \{\kbf_0\}.
\end{equation}

Fix $\mathbf{k} \in X_{qn} \setminus \{\mathbf{k}_0\}$. When $\kbf \notin q \NN^p$, $(\star\star)$ is precisely the conclusion of Lemma \ref{NotMultsOfp^2}, so we may assume that $\kbf \in q \NN^p$. Write $\mathbf{k} = q\mathbf{m}$ for some $\mathbf{m} \in \NN^p$, so that $|\mathbf{m}| = \frac{1}{q}|\mathbf{k}| = \frac{qn}{q} = n$. We first consider the case where $m_i < n$ for all $i$, so that by the inductive hypothesis we have $\vp(u_{p m_i}) = w(m_i)/p$. Suppose that $\vp(u_{pm_i}) > 1$ for some $i$. Then by Corollary \ref{coro14}(2) and Proposition \ref{prop17}(1), 
\[ \vp( u_{\kbf}) \geq \vp(u_{k_i}) = \vp(u_{qm_i}) > \frac{1}{p} > \frac{1}{q-1} > w(n) - 1\]
and ($\star\star$) holds. Otherwise, $\vp(u_{pm_i}) \leq 1$ for all $i$ and then by Corollary \ref{coro14}(1) we have 
\[\vp(u_{k_i}) = \vp(u_{qm_i}) = \frac{1}{p} \vp(u_{pm_i}) = \frac{1}{q}w(m_i).\] 
Since $|\mathbf{m}|=n$, Proposition \ref{prop17}(6) gives
\[\vp( u_{\kbf}) \geq \frac{1}{q} \sum w(m_i) \geq \frac{1}{q} \cdot w(n)  > w(n)-1\]
because $w(n) < 1 + 1/(q-1)$ by Proposition \ref{prop17}(1). Hence ($\star\star$) follows. 

We're left with the case where at least one $m_i$ is equal to $n$. But then since $|\mathbf{m}|=n$, all other $m_j$'s are zero and such $\mathbf{m}$'s form a single $S_p$-orbit of size $p$. Hence we have to show $(\star\star)$ holds when $\kbf = (0,0,\cdots,qn)$.

The congruence $(\diamond)$ together with our estimates above implies
\[ \vp(u_n-(u_{n p}^p + p u_{n q})) > w(n). \] 
Now, $u_{n p} \equiv u_{n q}^p \bmod{p}$ by Corollary \ref{ppm} so that $u_{np}^p \equiv u_{n q}^{q} \bmod{p^2} $. Therefore  \[ \vp(u_n-(u_{n q}^{q} + p u_{n q})) > w(n). \] Since we already know that $\vp(u_n) = w(n)$, we get that 
\[\vp(u_{n q}^{q} + p u_{n q}) = w(n).\]

We will now see that $\vp(pu_{nq}) \leq w(n)$ is not possible. Indeed, if $\vp(p u_{n q}) = w(n)$, then $\vp(u_{n q}^{q}) \geq w(n)$ so that $\vp(u_{n q}) \geq w(n)/q$ and $\vp(p u_{n q}) \geq 1 + w(n)/q > w(n)$. And if $\vp(p u_{n q}) < w(n)$ then $\vp(p u_{n q}) = \vp(u_{n q}^{q})$, so $\vp(u_{n q}) = 1/(q-1)$. But then $\vp(pu_{nq}) > 1 + 1/(q-1) > w(n)$ by Proposition \ref{prop17}(1). 

Hence $\vp(pu_{nq}) > w(n)$ after all, which is $(\star\star)$ for $\kbf = (0,0,\cdots,0,qn)$.
\end{proof}

%\bibliographystyle{amsalpha}
%\bibliography{ValPol}
\end{document}